\documentclass[a4paper,12pt]{article}
\usepackage{amsmath}
\usepackage{latexsym}
\usepackage{url}
\usepackage{color}
\numberwithin{equation}{section}

\def\R{{\bf R}}

\def\N{{\bf N}}

\def\d{\displaystyle}
\def\e{{\varepsilon}}

\def\wt{\widetilde}

\newtheorem{thm}{Theorem}[section]

\newtheorem{lem}{Lemma}[section]

\newtheorem{rem}{Remark}[section]

\title{A revisit via slicing method\\
on a quadratic semilinear wave equation\\
in two space dimensions
}
\author{
Masakazu Kato
\footnote{Laboratory of Mathematical Science, Graduate School of Science, University of Hyogo, 2167 Shosha, 
Himeji, Hyogo 671-2201, Japan.
e-mail: kato@sci.u-hyogo.ac.jp.}
, Hiroyuki Takamura
\footnote{Mathematical Institute,
Tohoku University,
 Aoba, Sendai 980-8578, Japan.
e-mail: hiroyuki.takamura.a1@tohoku.ac.jp.}
and
Kyouhei Wakasa
\footnote{
College of Liberal Arts, Mathematical Science Research Unit, 
Muroran Institute of Technology, 27-1, Mizumoto-cho, 
Muroran, Hokkaido 050-8585, Japan.
e-mail: wakasa@muroran-it.ac.jp
}}
\date{
\[
\begin{array}{ll}
\mbox{\footnotesize{\bf Keywords:}}
& \mbox{\footnotesize semilinear wave equation, two space dimensions, lifespan}\\
\mbox{\footnotesize{\bf MSC2020:}}
& \mbox{\footnotesize primary 35L71, secondary 35B44}\\
\end{array}
\]
}

\begin{document}
\maketitle
\begin{abstract}
In this paper,
we are focusing on the proof of the blow-up result
for a quadratic semilinear wave equation in two space dimensions.
There is a logarithmic loss in estimating the lifespan of classical solutions
if the 0th moment of the initial speed does not vanish.
This result is already known with almost sharp constants.
But in order to have a direct application to numerical analysis,
we show a simple proof by iteration argument of point-wise estimate of the solution
with the slicing technique.
\end{abstract}


\section{Introduction}
\par
We study the following initial value problem for semilinear wave equations
in two space dimensions
with unknown function $u:(x,t)\in \R^2\times[0,T)\rightarrow\R$.
\begin{equation}
\label{IVP}
\left\{
\begin{array}{l}
u_{tt}-\Delta u=|u|^p,\quad \mbox{in}\quad \R^2\times(0,T),\\
u(x,0)=\e f(x),\ u_t(x,0)=\e g(x),
\end{array}
\right.
\end{equation}
where  $T>0$, $p>1$ and $\e>0$ is a small parameter.
The initial data $f$ and $g$ are smooth functions of compact support. 
This problem plays a key role in establishing the optimality
of the general theory of nonlinear wave equations.
For a comprehensive overview and references concerning all spatial dimensions,
we refer the reader to Takamura \cite{Takamura_arXiv}.
\par
In this paper, we are interested in the lifespan estimates
for solutions to (\ref{IVP}).
The lifespan $T(\e)$ is defined by
\[
T(\e)=\sup\{T>0\ :\ \exists\ \mbox{a solution $u(x,t)$ of (\ref{IVP})
for arbitrarily fixed $(f,g)$.}\}.
\]
Throughout this paper, the term \lq\lq solution" means classical one when $p\ge2$,
and $C^1$ solution of the associated integral equations to (\ref{IVP}) when $1<p<2$.
We note that
Glassey \cite{G81a, G81b} proved that
$T(\e)<\infty$ for $1<p<p_S(2)$ and $T(\e)=\infty$ for $p>p_S(2)$.
Later, Schaeffer \cite{Sc85} proved that $T(\e)<\infty$ for $p=p_S(2)$.
Here,
\[
p_S(2)=\frac{3+\sqrt{17}}{2}
\]
is the so-called Strauss exponent which is a positive root of
\begin{equation}
\label{gamma}
\gamma(2,p):=1+\frac{3}{2}p-\frac{1}{2}p^2=0.
\end{equation}
In this sense, $p_S(2)$ is the critical exponent of the problem (\ref{IVP}).
This phenomenon is regarded as two dimensional case of Strauss conjecture.
See Strauss \cite{St81} in general dimensions.
\par
In the blow-up case, i.e. $1<p\le p_S(2)$,
we are interested in estimating of the lifespan $T(\e)$ to
clarify the stability of the trivial solution since the solution is unique.
We write $T(\e)\sim A(\e,C)$ if there are positive constants,
$C_1$ and $C_2$, independent of $\e$ satisfying $A(\e,C_1)\le T(\e)\le A(\e,C_2)$.
\par
When $1<p<p_S(2)$, we have the following results.
\begin{equation}
\label{lifespan_2d}
T(\e)\sim
\left\{
\begin{array}{cl}
C\e^{-(p-1)/(3-p)} & \mbox{for $1<p<2$ and $\d\int_{\R^2}g(x)dx\neq0$,}\\
Ca(\e)
&\mbox{for $p=2$ and $\d\int_{\R^2}g(x)dx\neq0$},\\
C\e^{-p(p-1)/\gamma(2,p)} & \mbox{for $1<p\le2$ and $\d\int_{\R^2}g(x)dx=0$},\\
&\mbox{\quad or, $2<p<p_S(2)$},
\end{array}
\right.
\end{equation}
where $a=a(\e)$ is a number satisfying 
\begin{equation}
\label{a}
a^2\e^2\log(1+a)=1.
\end{equation}
We note that $\gamma(2,2)=2$ and
\[
\frac{p-1}{3-p}<\frac{p(p-1)}{\gamma(2,p)}
\Longleftrightarrow 1<p<2
\]
which means that the first and second cases produce shorter lifespans than the third one in (\ref{lifespan_2d}).
This estimate is due to Lindblad \cite{L90} for $p=2$ with almost sharp constants
and  Zhou \cite{Z93} for $1<p<2$ with zero 0th moment of $g$.
More precisely, \cite{L90} obtained that
\begin{equation}
\label{lifespan_lim}
\left\{
\begin{array}{ll}
\d \exists\lim_{\e\rightarrow+0}a(\e)^{-1}T(\e)>0
& \mbox{when $\d\int_{\R^2}g(x)dx\neq0$},\\
\d\exists\lim_{\e\rightarrow+0}\e T(\e)>0
& \mbox{when $\d\int_{\R^2}g(x)dx=0$},
\end{array}
\right.
\quad\mbox{for $p=2$}.
\end{equation}
The case of $1<p<2$ with non-zero 0th moment of $g$ is due to
Takamura \cite{Takamura15} for the upper bound,
and Imai, Kato, Takamura and  Wakasa \cite{IKTW19} for the lower bound.
We also note that
the non-zero 0th moment case, or zero 0th moment case are categorized in
heat-like estimate, or wave-like estimate respectively.
The heat-like case is closely related to Fujita exponent in two space dimension $p_F(2)=2$
which is the critical exponent of semilinear heat equations.
This observation has been discussed in analysis on semilinear damped wave equations.
See Lai, Schiavone and Takamura \cite{LST20} for its details and references.
When $p=p_S(2)$,
Zhou \cite{Z93} proved that
\begin{equation}
\label{lifespan_critical}
T(\e)\sim\exp\left(C\e^{-p(p-1)}\right).
\end{equation}
\par
The main purpose of this article is to prove the upper bound of $T(\e)$ in the second case of (\ref{lifespan_2d})
only by point-wise estimates and an iteration argument in preparation for future work on  \lq\lq blow-up boundary"
of solutions and its numerical analysis.
Takamura \cite{Takamura_arXiv} recently reproved the upper bound
in the critical case (\ref{lifespan_critical}) in this direction.
The blow-up boundary was first introduced by Caffarelli and Friedman \cite{CF85, CF86}.
See also Sasaki \cite{Sasaki18} for derivative type nonlinearities with numerical simulations.
Moreover, Matsuya \cite{Matsuya13} first introduced discrete semilinear wave equations,
and Tsubota, Higashi, Matsuya, Sasaki and Tokihiro \cite{THMST} recently studied
the lifespan estimate for the discrete case in one space dimension.
\par
The reason to consider the special case of $p=2$ and non-zero 0th moment of initial speed
in this paper is that it is the critical case of heat-like lifespan estimates.
Moreover, only this case cannot be handled by functional method developed
by Ikeda, Sobajima and Wakasa \cite{ISW19}
which based on the weak form with a special choice of the test function
and has wide applications to many different kinds of equations.
We note that when iteration argument introduced by John \cite{J79} is applied
to the critical case in which the solution has logarithmic growth in time,
the slicing technique by Agemi, Kurokawa and Takamura \cite{AKT00} is required.
Such a technique is widely applicable to various inequalities with critical exponents
including integral inequalities drived from ordinary differential inequalities of the second order.
For a recent example, see Sasaki, Shao and Takamura \cite{SST25} 
where critical blow-up results for power-type nonlinearities involving spatial derivatives are established.


\section{Theorem and its proof by slicing method}
As in Introduction, the purpose of this paper is to reprove the following theorem
by means of an iteration argument combined with the slicing technique.

\begin{thm}
\label{thm:main}
Let $p=2$.
Assume that $f\in C_0^3(\R^2)$ and $g\in C_0^2(\R^2)$ satisfy $\d \int_{\R^2} g(x)dx>0$.
Then, there is a positive constant $\e_0=\e_0(f,g)$ such that
a classical solution of (\ref{IVP}) cannot exist as far as $T$ satisfies
\[
T> Ca(\e)
\quad{for}\ 0<\e\le\e_0,
\]
where $a$ is defined in (\ref{a}) and $C$ is a positive constant independent of $\e$.
\end{thm}

\begin{rem}
This theorem is weaker version of the result of Lindblad \cite{L90}.
In view of (\ref{lifespan_lim}), \cite{L90} established the lifespan estimate with a specific constant $C$
both from below and above.
The proof of Takamura \cite{Takamura15} is due to a combination
with the improved Kato's lemma on ordinary differential inequalities and 
a point-wise estimate of a solution to the free wave equation,
but it requires the stronger assumption that $f(x)\equiv0$ and $g(x)\ge0(\not\equiv0)$.
We note that  the conclusion shows that
\[
T(\e)\le Ca(\e)
\quad{for}\ 0<\e\le\e_0
\]
with the same $\e_0$ and $C$.
\end{rem}

From now on, we may assume that
\begin{equation}
\label{supp_initial}
\mbox{supp}\ (f,g)\subset\{x\in\R^2\ :\ |x|\le k\},\ k\ge1
\end{equation}
without loss of generality.
By finite propagation speed of the nonlinear wave,
(\ref{supp_initial}) implies that
\begin{equation}
\label{supp_sol}
\mbox{supp}\ u\subset\{(x,t)\in\R^2\times[0,T]\ :\ |x|\le t+k\}.
\end{equation}
In order to prepare the tools needed for the proof,
we introduce the integral equation equivalent to the classical solution of (\ref{IVP})
with $p=2$,
\begin{equation}
\label{integral}
u=\e u^0+L(u^2)\quad\mbox{in}\ \R^2\times[0,T),
\end{equation}
where $u^0=u^0(x,t)$ is the solution to the free wave equation with the initial data $(f,g)$
and $L$ is Duhamel's term,
\[
L(v)(x,t):=\frac{1}{2\pi}\int_0^tds\int_{|x-y|\le t-s}\frac{v(y,s)}{\sqrt{(t-s)^2-|x-y|^2}}dy.
\]
For a function $u=u(x,t)$, define a spherical mean of $u$ at $x=0$ with radius $r\geq0$ defined by
\[
\wt{u}(r,t):=\frac{1}{2\pi}\int_{|\omega|=1}u(r\omega,t)dS_\omega.
\]

\par
First, let us employ the following lemma which will be a frame of the iteration. 

\begin{lem}[\cite{Agemi91}]
Let $p=2$.
Assume that $(f,g)\in C_0^3(\R^2)\times C_0^2(\R^2)$ satisfies (\ref{supp_initial}).
Then, a classical solution of (\ref{IVP}) satisfies
\begin{equation}
\label{basic}
\wt{u}(r,t)\ge\e\wt{u^0}(r,t)+
\frac{1}{2\pi\sqrt{r}}\iint_{T_{r,t}}\sqrt{\lambda}\wt{u}(\lambda,\tau)^2d\lambda d\tau
\quad\mbox{for}\ k\le t-r\le r,
\end{equation}
where a domain $T_{r,t}$ is defined by
\[
T_{r,t}:=\{(\lambda,\tau)\in(0,\infty)^2:t-r\le\lambda,\ \tau+\lambda\le t+r,\ k\le\tau-\lambda\le t-r\}.
\]
\end{lem}

\par\noindent
{\bf Proof.} See (2.12) in Agemi \cite{Agemi91} which doesn't have $\e\wt{u^0}$,
but it is trivial to add it in view of (\ref{integral}).
\hfill$\Box$

\vskip10pt
To have the first step of the iteration argument,
we need the following estimate for solution to the free wave equation.

\begin{lem}[\cite{Hormander97, L90, IKTW19}]
\label{lem:linear}
Assume that $(f,g)\in C_0^3(\R^2)\times C_0^2(\R^2)$ satisfies (\ref{supp_initial}).
Then, there exist positive constants $D_1,D_2$ such that
\begin{equation}
\label{est:linear}
\begin{array}{l}
\d\left|u^0(x,t)-\frac{D_1}{\sqrt{t+|x|+2k}\sqrt{t-|x|+2k}}\int_{\R^2}g(x)dx\right|\\
\d\le\frac{D_2}{\sqrt{t+|x|+2k}(t-|x|+2k)^{3/2}}
\end{array}
\end{equation}
in $\{(x,t)\in\R^2\times[0,\infty)\ :\ |x|\le t+k\}$.
\end{lem}

\begin{rem}
This lemma first appeared in H\"{o}rmander \cite{Hormander97} on a partial space-time domain.
Later, Lindblad \cite{L90} refined it.
Finally, Imai, Kato, Takamura and Wakasa \cite{IKTW19} extended it to the full domain.
\end{rem}

\par\noindent
{\bf Proof of Lemma \ref{lem:linear}.} In \cite{IKTW19},
only an upper estimate of $|u^0(x,t)|$ was derived.
But, in view of (2.5) in \cite{IKTW19},
the present lemma can be proved by exactly the same argument as that used in the proof of Lemma 2.1 in \cite{IKTW19}.
\hfill$\Box$
\vskip10pt

\par\noindent
{\bf Proof of Theorem \ref{thm:main} by slicing the blow-up set.}
We note that (\ref{supp_sol}) implies that
\[
\mbox{supp}\ \wt{u}\subset\{(r,t)\in[0,\infty)\times[0,T]\ :\ r\le t+k\}.
\]

\par
It follows from Lemma \ref{lem:linear} and the assumption $\d \int_{\R^2}g(x)dx>0$ 
that there exist positive constants $B=B(f,g)$ and $M=M(f,g)$ such that
\begin{equation}
\label{est:0}
\wt{u^0}(r,t)\ge\frac{B}{\sqrt{t+r}\sqrt{t-r}}\quad\mbox{for}\ M\le t-r\le r.
\end{equation}
Without loss of generality, we may assume that
\[
M\ge2k.
\] 
Introducing characteristic variables by
\[
\alpha:=\tau+\lambda,\ \beta:=\tau-\lambda
\]
into the right-hand side of (\ref{basic}) and using the inequality
\[
t+r\ge 3(t-r)
\]
which is equivalent to $t-r\le r$, we have that
\begin{equation}
\label{frame}
\wt{u}(r,t)\ge\frac{D}{\sqrt{t+r}}\int_{l_jM}^{t-r}d\beta
\int_{2(t-r)+\beta}^{3(t-r)}\sqrt{\alpha-\beta}\wt{u}(\lambda,\tau)^2d\alpha
\quad\mbox{in}\ \Sigma_j,
\end{equation}
where $D:=(4\sqrt{2}\pi)^{-1}>0$ and \lq\lq sliced" domains $\Sigma_j\ (j\in\N\cup\{0\})$ are defined by
\[
\Sigma_j:=\{(r,t)\in(0,\infty)\times[0,T)\ :\ l_jM\le t-r\le r\}
\]
with a \lq\lq slicer"
\[
l_j:=\sum_{i=0}^j\frac{1}{2^i}.
\]
We note that
\[
\Sigma_{j+1}\subset\Sigma_j\quad\mbox{and}\quad 1\le l_j<2
\]
hold for all $j\in\N\cup\{0\}$.
\par
The first estimate of the iteration argument is established as below.
Using the trivial inequality $\wt{u}\ge\e\wt{u^0}$
and substituting the lower bound in  (\ref{est:0}) into the right-hand side of (\ref{frame}) with $j=0$,
we have that
\[
\begin{array}{ll}
\wt{u}(r,t)&
\d\ge\frac{DB^2\e^2\sqrt{2(t-r)}}{\sqrt{t+r}}\int_{l_0M}^{t-r} d\beta \ \frac{1}{\beta}
\int_{2(t-r)+\beta}^{3(t-r)}\frac{1}{\alpha}d\alpha\\
&\d\ge\frac{\sqrt{2}DB^2\e^2}{3\sqrt{t+r}\sqrt{t-r}}
\int_{l_0M}^{t-r}\frac{t-r-\beta}{\beta}d\beta
\end{array}
\]
in $\Sigma_0$.
Then, the $\beta$-integral in the last line above can be estimated as follows.
Integration by parts yields 
\[
\begin{array}{ll}
\d\int_{l_0M}^{t-r}\frac{t-r-\beta}{\beta}d\beta
&\d=\int_{l_0M}^{t-r}(t-r-\beta)\left(\log\frac{\beta}{l_0M}\right)'d\beta\\
&\d=\int_{l_0M}^{t-r}\log\frac{\beta}{l_0M}d\beta.
\end{array}
\]
Hence it follows from $\Sigma_1\subset\Sigma_0$ and
\[
\frac{l_0}{l_1}(t-r)\ge l_0M\quad\mbox{for}\ (r,t)\in\Sigma_1
\]
that
\[
\begin{array}{ll}
\d\int_{l_0M}^{t-r}\frac{t-r-\beta}{\beta}d\beta
&\d\ge\int_{l_0(t-r)/l_1}^{t-r}\log\frac{\beta}{l_0M}d\beta\\
&\d\ge\left(1-\frac{l_0}{l_1}\right)(t-r)\log\frac{t-r}{l_1M}
\end{array}
\]
in $\Sigma_1$ which yields that
\begin{equation}
\label{est:1}
\wt{u}(r,t)\ge\frac{C_1}{\sqrt{t+r}}\left(\frac{t-r}{l_1M}\right)^{1/2}\log\frac{t-r}{l_1M}
\quad\mbox{in}\ \Sigma_1,
\end{equation}
where
\[
C_1:=\frac{\sqrt{2}DB^2}{3^2}\e^2>0.
\]
This estimate illustrates the essence of the slicing method.

\par
Now, in order to employ mathematical induction, we assume an estimate
\begin{equation}
\label{est:j}
\wt{u}(r,t)\ge\frac{C_j}{\sqrt{t+r}}\left(\frac{t-r}{l_jM}\right)^{a_j}\left(\log\frac{t-r}{l_jM}\right)^{b_j}
\quad\mbox{in}\ \Sigma_j,
\end{equation}
where all the positive constants $a_j,b_j,C_j$ are defined later.
But in view of (\ref{est:1}),  this estimate is true for $j=1$ with
\begin{equation}
\label{1}
a_1=\frac{1}{2},\quad b_1=1,\quad C_1=\frac{\sqrt{2}DB^2}{3^2}\e^2.
\end{equation}
Substituting the estimate (\ref{est:j}) into the right-hand side of (\ref{frame}), we obtain
\[
\begin{array}{ll}
\wt{u}(r,t)&
\d\ge\frac{DC_j^2\sqrt{2(t-r)}}{\sqrt{t+r}}
\int_{l_jM}^{t-r}d\beta \left(\frac{\beta}{l_jM}\right)^{2a_j}\left(\log\frac{\beta}{l_jM}\right)^{2b_j}
\int_{2(t-r)+\beta}^{3(t-r)}\frac{1}{\alpha}d\alpha\\
&\d\ge\frac{\sqrt{2}DC_j^2}{3\sqrt{t+r}\sqrt{t-r}}
\int_{l_jM}^{t-r}(t-r-\beta)\left(\frac{\beta}{l_jM}\right)^{2a_j}\left(\log\frac{\beta}{l_jM}\right)^{2b_j}d\beta
\end{array}
\]
in $\Sigma_j$.
Hence it follows from $\Sigma_{j+1}\subset\Sigma_j$ and
\[
\frac{l_j}{l_{j+1}}(t-r)\ge l_jM\quad\mbox{for}\ (r,t)\in\Sigma_{j+1}
\]
that the $\beta$-integral in the last line above can be estimated in $\Sigma_{j+1}$
by applying the slicing method again,
\[
\begin{array}{l}
\d\int_{l_jM}^{t-r}(t-r-\beta)\left(\frac{\beta}{l_jM}\right)^{2a_j}\left(\log\frac{\beta}{l_jM}\right)^{2b_j}d\beta\\
\d\ge\int_{l_j(t-r)/l_{j+1}}^{t-r}(t-r-\beta)\left(\frac{\beta}{l_jM}\right)^{2a_j}\left(\log\frac{\beta}{l_jM}\right)^{2b_j}d\beta\\
\d\ge\left(\frac{t-r}{l_{j+1}M}\right)^{2a_j}\left(\log\frac{t-r}{l_{j+1}M}\right)^{2b_j}
\int_{l_j(t-r)/l_{j+1}}^{t-r}(t-r-\beta)d\beta\\
\d\ge\frac{1}{2}\left(1-\frac{l_j}{l_{j+1}}\right)^2\left(\frac{t-r}{l_{j+1}M}\right)^{2a_j}
\left(\log\frac{t-r}{l_{j+1}M}\right)^{2b_j}(t-r)^2\\
\d\ge\frac{(t-r)^2}{2^{2j+5}}
\left(\frac{t-r}{l_{j+1}M}\right)^{2a_j}\left(\log\frac{t-r}{l_{j+1}M}\right)^{2b_j}
\end{array}
\]
since
\[
1-\frac{l_j}{l_{j+1}}=\frac{1}{2^{j+1}l_{j+1}}\ge\frac{1}{2^{j+2}}.
\]
Summing up, we obtain that
\begin{equation}
\label{j+1}
\wt{u}(r,t)\ge\frac{\sqrt{2}DC_j^2}{2^{2j+5}\cdot3\sqrt{t+r}}\left(\frac{t-r}{l_{j+1}M}\right)^{2a_j+\frac{3}{2}}
\left(\log\frac{t-r}{l_{j+1}M}\right)^{2b_j}
\end{equation}
in $\Sigma_{j+1}$.

\par
Now we are in a position to define sequences $\{a_j\},\{b_j\}$ by
\[
a_{j+1}=2a_j+\frac{3}{2},\quad b_{j+1}=2b_j,
\]
which can be solved by (\ref{1}) as
\[
a_j=2^j-\frac{3}{2},\quad b_j=2^{j-1}\quad(j\in\N).
\]
Hence it follows from (\ref{j+1}) that
\[
\wt{u}(r,t)
\ge\frac{\sqrt{2}DC_j^2}{2^{2j+5}\cdot3\sqrt{t+r}}
\left(\frac{t-r}{l_{j+1}M}\right)^{2^{j+1}-\frac{3}{2}}\left(\log\frac{t-r}{l_{j+1}M}\right)^{2^j}
\]
in $\Sigma_{j+1}$.
So we have to set
\[
C_{j+1}=\frac{E}{4^j}C_j^2,
\]
where
\[
E:=\min\left\{1, \frac{\sqrt{2}D}{2^5\cdot3}\right\}>0.
\]
This recurrence relation yields
\[
C_{j+1}=\frac{E^{1+2+\cdots+2^{j-1}}C_1^{2^j}}{4^{j+2(j-1)+\cdots+2^{j-1}}},
\quad
C_1=\frac{\sqrt{2}DB^2}{3^2}\e^2.
\]
We note that
\[
j+2(j-1)+\cdots+2^{j-1}=2^{j-1}\left(\frac{j}{2^{j-1}}+\frac{j-1}{2^{j-2}}+\cdots+1\right)
\]
and there is a positive constant $S$ such that
\[
\frac{j}{2^{j-1}}+\frac{j-1}{2^{j-2}}+\cdots+1\nearrow S\quad(j\rightarrow\infty)
\]
by the d'Alembert criterion.

\par
Therefore we obtain that
\[
\wt{u}(r,t)\ge\frac{E^{-1}}{\sqrt{t+r}}\left(\frac{t-r}{2M}\right)^{-3/2}\exp\left(2^j \log I(r,t)\right)
\quad\mbox{in}\ \Sigma_\infty
\]
for all $j\in\N$, where we set
\[
\Sigma_\infty:=\{(r,t)\in(0,\infty)\times[0,T)\ :\ 2M\le t-r\le r\}
\]
and
\[
I(r,t):=\frac{E C_1}{4^{S/2}}\left(\frac{t-r}{2M}\right)^2\log\frac{t-r}{2M}.
\]
If there is a point $(r_0,t_0)\in\Sigma_\infty$ such that $I(r_0,t_0)>1$,
then $\wt{u}$ as well as $u$ cannot be a classical solution of (\ref{IVP}) by letting $j\rightarrow\infty$.
A sufficient condition for the existence of such a point $(r_0,t_0)$ is
\[
I(r_0,t_0)=B_1t_0^2 \e^2 \log\frac{t_0}{4M}>1
\]
by setting $t_0=2r_0$ and $\displaystyle  B_1:=\frac{\sqrt{2}DB^2E}{4^{S/2+2} \cdot 3^2M^2}$.
\par
Now we set
\begin{align*}
 t_0=B_2^{-1} a(\e) \quad \mbox{and} \quad a(\e_0)=(8M)^2,
\end{align*}
where $B_2=\min \{1, (B_1/4)^{1/2} \}$ and note that
$a(\e)$ is monotonously decreasing in $\e$ and 
$\displaystyle \lim_{\e\to+0}a(\e)=+\infty$.
Hence, for $0< \e \leq \e_0$, we obtain that
\begin{align*}
	t_0\geq (8M)^2 \quad \mbox{i.e.} \quad \frac{1}{2}\log t_0\geq \log 8M,
\end{align*}
so that
\begin{align*}
 I(t_0/2,t_0)&\geq B_1 t_0^2 \e^2 \log \frac{t_0+1}{8M}\\
 &\geq 2 B_2^2 t_0^2\e^2\log(t_0+1)\\
 &\geq 2a(\e)^2 \e^2 \log\left(a(\e)+1 \right)=2.
\end{align*}
Therefore $u$ cannot exist as far as $T$ satisfies that
\[
T>B_2^{-1}a(\e)\quad\mbox{for}\ 0<\e\le\e_0.
\]
This completes the proof of Theorem \ref{thm:main}.
\hfill$\Box$

\begin{rem}
In view of the proof of the first step of the iteration,
there is no logarithmic term in (\ref{est:1}) if we have an estimate such as
\[
\wt{u^0}(r,t)\ge\frac{B}{\sqrt{t+r}(t-r)^{3/2}}\quad\mbox{in}\ \Sigma_0
\]
when
\[
\int_{\R^2}g(x)dx=0
\]
instead of (\ref{est:0}).
This estimate provides the upper bound in the third case of (\ref{lifespan_2d}).
\end{rem}

\vskip10pt
\par\noindent
{\bf Acknowledgment.} 
All the authors are partially supported by the Grant-in-Aid for Scientific Research (A) (No.22H00097),
Japan Society for the Promotion of Science.
Moreover, the second author is partially supported by
the Grant-in-Aid for Scientific Research (C) (No.24K06819),
Japan Society for the Promotion of Science, and
the Science Committee of the Ministry of Science and Higher Education of the Republic of Kazakhstan (AP26195417).
The third author is partially supported by the Grant-in-Aid for Scientific Research (C) (No.25K07093),
Japan Society for the Promotion of Science.
Finally, all authors sincerely thank the reviewer for the numerous insightful comments
and helpful suggestions, which have greatly improved the manuscript.


\bibliographystyle{plain}

\end{document}